\documentclass{amsart}
\usepackage[utf8]{inputenc}
\usepackage{enumerate}%%序号
\usepackage{amsmath}%%数学公式
\usepackage{cite}%%%%%引用
\usepackage{amsfonts}%%数学字体
\usepackage{mathrsfs}%%数学字体
\usepackage{galois}%%%%一些数学符号的改进
\usepackage[toc,page]{appendix}%%%Appendix
\usepackage{xcolor}
\usepackage[colorlinks, linkcolor = blue, anchorcolor = blue, citecolor = blue]{hyperref}
\usepackage{eso-pic}
\usepackage{enumitem}
\usepackage{amssymb}

\def\acts{\curvearrowright} %group action symbol

\newtheorem{proposition}{Proposition}[section]
\newtheorem{claim}{Claim}
\newtheorem{corollary}{Corollary}[section]
\newtheorem{theorem}{Theorem}[section]

\newtheorem{lemma}{Lemma}[section]

\newtheorem{conjecture}{Conjecture}

\theoremstyle{definition}
\newtheorem{definition}{Definition}[section]
\newtheorem{remark}{Remark}[section]

\numberwithin{equation}{section}

\newcommand{\set}[1]{\{#1\}}

\newcommand{\Met}{\mathrm{Met}}
\newcommand{\Diff}{\mathrm{Diff}}
\newcommand{\PSC}{\mathrm{PSC}}
\newcommand{\Isom}{\mathrm{Isom}}
\newcommand{\CC}{\mathrm{CC}}
\title{Equivariant $3$-manifolds with positive scalar curvature}

\author{Tsz-Kiu Aaron Chow}
\address{Department of Mathematics, Columbia University}
\email{achow@math.columbia.edu}
\author{Yangyang Li}
\address{Department of Mathematics, Princeton University, Fine Hall, 304 Washington Road, Princeton, NJ 08540, USA}
\email{yl15@math.princeton.edu}

\begin{document}
\bibliographystyle{abbrvalpha}
\begin{abstract}
    In this paper, for any compact Lie group $G$, we show that the space of $G$-invariant Riemannian metrics with positive scalar curvature (PSC) on any closed three-manifold is either empty or contractible. In particular, we prove the generalized Smale conjecture for spherical three-orbifolds. Moreover, for connected $G$, we make a classification of all PSC $G$-invariant three-manifolds.   
\end{abstract}

\maketitle

\section{Introduction}

    For a connected, orientable, smooth manifold $M$, we denote the diffeomorphism group of $M$ by $\Diff(M)$. When $G$ is a compact Lie group which is a subgroup of $\Diff(M)$, we denote by $\Met(M, G)$ the space of $G$-invariant Riemannian metrics and by $\Met_{\PSC}(M, G)$ the space of $G$-invariant metrics with positive scalar curvature. For any $g \in \Met(M, G)$, we will call $(M, g)$ a \textbf{$G$-invariant Riemannian manifold}. Note that we always assume the action of $G$ on $M$ to be effective. 

    The main goal of this paper is to study the space $\Met_{\PSC}(M, G)$ for closed three-dimensional manifold $M$, and the topology of $M$ itself when $\Met_{\PSC}(M, G) \neq \emptyset$.

    When $G = \set{\mathrm{Id}}$, by using Ricci flow together with Gromov-Lawson surgery techniques \cite{gromov_classification_1980}  F. Cod\'a Marques \cite{marques_deforming_2012} first proved that if $\Met_{\PSC}(M) = \Met_{\PSC}(M, \set{Id})$ is nonempty, then $\Met_{\PSC}(M)/\Diff(M)$ is path-connected. In particular, $\mathcal{R}_{+}(S^3)$ is path-connected due to J. Cerf's path-connectedness result \cite{cerf_sur_1968} of $\mathcal{R}_+(S^3)$. Later, Bessi{\`e}res-Besson-Maillot-Marques \cite{bessieres_deforming_2020} showed that the path-connectedness holds for complete manifolds with bounded geometry and uniformly positive scalar curvature. Interested readers may refer to Carlotto-Li \cite{carlotto_constrained_2019} and Hirsch-Lesourd \cite{hirsch_moduli_2021} for similar results of three-manifolds with boundary or asymptotically flat three-manifolds with boundary.

    Recently, R. Bamler and B. Kleiner \cite{bamler_ricci_2019} made a remarkable progress and showed that $\Met_{\PSC}(M)$ is either empty or contractible, whose proof relies on the existence and uniqueness of singular Ricci flow, a canonical Ricci flow proposed by J. Lott and B. Kleiner \cite{kleiner_singular_2017,bamler_uniqueness_2018}. Moreover, their arguments also lead to an affirmative answer to the generalized Smale conjecture for spherical space forms.

    \begin{conjecture}[Generalized Smale conjecture for spherical space forms]
        If $(M, g)$ is an isometric quotient of the round sphere, then the inclusion map $\Isom(M, g) \hookrightarrow \Diff(M, g)$ is a homotopy equivalence.
    \end{conjecture}

    Historically, the bijection on path components of $\Isom(S^3/\Gamma, g)$ and $\Diff(S^3/\Gamma, g)$ dates back to Cerf's works on $S^3$ \cite{cerf_nullite_1962,cerf_nullite_1962-1,cerf_nullite_1962-2,cerf_nullite_1962-3}. This bijection for general spherical space forms was confirmed through a number of works \cite{MR524878,rubinstein_3-manifolds_1979,10.1007/BFb0088106,bonahon_diffeotopies_1983,rubinstein_one-sided_1984,boileau_scindements_1991}. Based on Hatcher's impressive work \cite{hatcher_proof_1983} on $S^3$, \cite{ivanov_homotopy_1984,hong_diffeomorphisms_2012,bamler_ricci_2017} have proved the homotopy equivalence results for most of spherical space forms, before Bamler-Kleiner's resolution \cite{bamler_ricci_2019}.

    In this paper, we first extend the Bamler-Kleiner's contractibility result to the $G$-invariant setting.

    \begin{theorem}\label{thm:1}
        If $M$ is a connected sum of spherical space forms and copies of $S^2 \times S^1$, then $\Met_{\PSC}(M, G)$ is nonempty and contractible; Otherwise, $\Met_{\PSC}(M, G)$ is empty.
    \end{theorem}
    \begin{remark}
        When the group $G$ consists of $\mathrm{Id}$ and the reflection map, this was recently proved by A. Carlotto and C. Li \cite{carlotto_constrained_2019,carlotto_constrained_2021}, based on the techniques in F. Cod\'a Marques \cite{marques_deforming_2012} and Bamler-Kleiner \cite{bamler_ricci_2019}.
    \end{remark}

    Note that metric spaces mentioned above, i.e., $\Met(M,G)$ and $\Met_{\PSC}(M, G)$, are all absolute neighborhood retracts (ANR) (See, for example, \cite[Section 3]{carlotto_constrained_2021}), so we only need to prove the weak homotopic equivalence, i.e., every $S^k$ continuous family of metrics in $\Met_{\PSC}(M, G)$ can bound a $B^{k+1}$ conitnuous family of metrics in $\Met_{\PSC}(M, G)$. Hence, Theorem \ref{thm:1} can reduce to the following deformation result, which is an equivariant version of \cite[Theorem~8.1]{bamler_ricci_2019}.

    \begin{theorem}\label{thm:2}
        If $M$ is a connected sum of spherical space forms and copies of $S^2 \times S^1$, then for any continuous family (indexed by a simplicial complex $K$) of metrics $(g^s)_{s \in K} \subset \Met(M, G)$, we have a family of metric deformations $(h^s_t)_{s \in K, t \in [0, 1]} \subset \Met(M, G)$ satisfying the following conditions.
        \begin{enumerate}
             \item $h^s_0 = g^s\,$;
             \item $(h^s_1)_{s \in K} \subset \Met_{\PSC}(M, G)$;
             \item If $g^s \in \Met_{\PSC}(M, G)$, then $h^s_t \in \Met_{\PSC}(M, G)$ for any $t \in [0, 1]$;
             \item In addition, if we assume that for a subcomplex $L \subset K$, $(M^s, g^s)$ is a CC-metric for all $s \in L$, then $h^s_t$ is a CC-metric for all $t \in [0,1]$. Here, a metric on a 3-manifold is called a \textbf{CC-metric} if it is homothetic to a quotient of the round sphere or round cylinder.
         \end{enumerate}
    \end{theorem}
    \begin{remark}
        In the following, the space of all $G$-invariant CC-metrics on $M$ will be denoted by $\Met_{\CC}(M, G)$.
    \end{remark}

    Our proof is based on Bamler-Kleiner's arguments. Modulo technical difficulties, their idea is intuitive for $G = \set{Id}$. First, the existence and uniqueness of singular Ricci flows (See Section \ref{sec:rf}) imply that there exists a continuous family of singular Ricci flows $\set{\mathcal{M}^s}_{s \in K}$. Then, thanks to the classifications of $\kappa$-solutions (See, for example, \cite{Bre18}, \cite{Bre19}, \cite[Section 3]{bamler_ricci_2019}), which are singularity models, one can apply a rounding process near singularities, such that they can be parametrized continuously. Similar to the proof in \cite{marques_deforming_2012}, for each fixed metric in the family, one could apply Gromov-Lawson gluing techniques on singularities by backward induction to obtain a deformation to a metric with positive scalar curvature. In their paper, Bamler-Kleiner introduced a novel partial homotopy argument so that one can perform such gluings continuously along $K$. 

    Similar to \cite[Theorem~1.3]{bamler_ricci_2019}, we can also prove an equivariant version of the Smale conjecture for $S^3$, and in particular, the generalized Smale conjecture for spherical $3$-orbifolds.

    \begin{corollary}\label{equiv smale conj}
        Given a compact Lie subgroup $G$ of $\Isom(S^3, g_{\mathrm{round}}) = \mathrm{SO}(4)$, let $\Diff(S^3, G)$ and $\Isom(S^3, G)$ denote the subgroups commuting with $G$ of $\Diff(S^3)$ and $\Isom(S^3, g_{\mathrm{round}})$, respectively. Then the inclusion map $\Isom(S^3, G) \hookrightarrow \Diff(S^3, G)$ is a homotopy equivalence. In particular, the generalized Smale conjecture for spherical $3$-orbifolds holds.
    \end{corollary}
    \begin{remark}
        Among others, Mecchia-Seppi \cite{mecchia_isometry_2019} recently established the bijection result of $\pi_0$'s for spherical $3$-orbifolds.
    \end{remark}

    Furthermore, for any connected compact Lie group $G$, combined with various classification works on $3$-manifolds with Lie group actions \cite{mostertCompactLieGroup1957,orlikActions3Manifolds1968,neumann3DimensionalGManifolds2Dimensional1968,alekseevskyRiemannianGmanifoldOnedimensionalOrbit1993} and etc., we can give a thorough list of possible manifolds $M$ admitting equivariant metrics with positive scalar curvature.

    \begin{theorem}\label{thm:3}
        If $\Met_{\PSC}(M, G) \neq \emptyset$ and $G$ is connected, then the oriented $3$ manifold $M$ can be classified with respect to the dimension of its principal orbits $G/H$ as follows:
        \begin{description}
            \item[$\dim(G/H) = 0$] $M$ is a connected sum of spherical space forms and copies of $S^2 \times S^1$. 
            $G$ is a trivial action.
            \item[$\dim(G/H) = 1$] If the set of fixed points $M^G = \emptyset$, then $M$ is either a spherical space form or $S^2 \times S^1$; Otherwise, $M$ is a connected sum of $S^3$, $L(p,q)$ and copies of $S^2 \times S^1$. 
            $G$ is $SO(2)$.
            \item[$\dim(G/H) = 2$] $M$ is $S^2 \times S^1$, a lens space or $RP^3 \sharp RP^3$. 
            $G$ is either $T^2$ or $SO(3)$ ($G \neq T^2$ when $M = RP^3 \sharp RP^3$).
            \item[$\dim(G/H) = 3$] $M$ is either $S^3$, $RP^3$, $L(m, 1)\ (m > 2)$, $S^3/D^*_m\ (m > 2)$, $S^3/T^*$, $S^3/O^*$ or $S^2 \times S^1$. 
            $G$ is a transitive Lie subgroup of the corresponding isometry group.
        \end{description}
    \end{theorem}
    
    Here is the organization of the paper.
    
    In Section 2 to Section 4, we verify $G$-invariance in the constructions of singular Ricci flows, rounding process and partial homotopies, so that we can adapt Bamler-Kleiner's work to our setting. 
    
    In Section 5, we complete the proof of our main theorem, and its corollary, i.e., the generalized Smale conjecture for spherical orbifolds. Among others, we need to prove the fibration $\Diff(S^n, G) \rightarrow \Met_{K\equiv 1}(S^n, G)$, which is closely related to the de Rham spherical rigidity theorem \cite{de_rham_reidemeisters_1964}.
    
    In Section 6, we classify all equivariant PSC 3-manifolds in terms of the group action.
    
    Appendix A contains some preparatory results about equivariant metrics, which will be essentially used in Section 4.

\section*{Acknowledgments}

The authors would like to thank Professor Richard Bamler for his interest in this project and inspiring conversations on the paper \cite{bamler_ricci_2019}. The first author would like to thank his advisor Professor Simon Brendle for his continuing support. The second author thanks his advisor Fernando Coda Marques for his constant support, Chao Li for telling him about his work joint with A. Carlotto \cite{carlotto_constrained_2021}, and Jonathan Zung for several helpful discussions.

\section{Equivariant singular Ricci flows}\label{sec:rf}

    As in the previous section, let $G$ be a compact Lie group which is a subgroup of $\Diff(M)$ acting on a Riemannian manifold $M$ via $\rho: G \acts M$. 

    \begin{definition}[Equivariant subset]
        A subset of $U$ of $G$ is said to be $G$-invariant if for any $h \in G$, we have
        \begin{equation}
            \rho_h.U = U\,.
        \end{equation}
    \end{definition}

    \begin{definition}[Equivariant function]
        A function $f$ on a $G$-invariant subset is said to be $G$-invariant if for any $h \in G$, we have
        \begin{equation}
            (\rho_h)^*(f) = f\,.
        \end{equation}
    \end{definition}

    \begin{definition}[Equivariant vector field]
        A vector field $X$ on a $G$-invariant subset $U$ is said to be $G$-invariant if for any $h \in G$, we have
        \begin{equation}
            (\rho_h)_*(X) = X\,.
        \end{equation}
    \end{definition}
    \begin{remark}
        By tensor contractions, we can easily extend the definition of equivariance to any tensors on an equivariant subset $U$.
    \end{remark}

    \begin{definition}[Equivariant singular Ricci flow]{\label{def equivariance singular ricci flow}}
        For a compact Lie group $G$, we say that $(\mathcal{M}, \mathfrak{t}, \partial_t, g) $ is a $G$-invariant singular Ricci flow if $(\mathcal{M}, \mathfrak{t}, \partial_t, g)$ is a singular Ricci flow (See \cite[Definition~3.15]{bamler_ricci_2019}) with the following properties.
        \begin{enumerate}[label = (\arabic*)]
            \item $\mathcal{M}_0$ is a closed three-manifold endowed with a $G$-invariant metric $g_0$, i.e., there exists an isometric group action $\rho: G \acts (\mathcal{M}_0, g_0)$;
            \item $\rho$ extends to a smooth group action $\hat\rho: G\acts \mathcal{M}$ such that $\hat \rho$ preserves each time slice $\mathcal{M}_t$ and acts on it isometrically;
            \item The vector field $\partial_{\mathfrak{t}}$ is $G$-invariant under the group action $\hat \rho$.
        \end{enumerate}
    \end{definition}

    We first recall the existence and uniqueness results of singular Ricci flows.

    \begin{theorem}[Existence {\cite[Theorem~1.1]{KL17}}]
        For every closed Riemannian three-manifold $(M, g)$, there exists a singular Ricci flow $(\mathcal{M}, \mathfrak{t}, \partial_{\mathfrak{t}}, g)$ with $(\mathcal{M}_0, g_0) \cong (M, g)$.
    \end{theorem}

    \begin{theorem}[Uniqueness {\cite[Theorem~1.3]{bamler_uniqueness_2018}}]\label{thm:unique}
        Any singular Ricci flow $(\mathcal{M}, \mathfrak{t}, \partial_{\mathfrak{t}}, g)$ is uniquely determined by its initial time-slice $(\mathcal{M}_0, g_0)$ up to isometry in the following sense:

        If $(\mathcal{M}, \mathfrak{t}, \partial_{\mathfrak{t}}, g)$ and $(\mathcal{M}, \mathfrak{t}, \partial_{\mathfrak{t}}, g)$ are two singular Ricci flows and $\phi: (\mathcal{M}_0, g_0) \rightarrow (\mathcal{M}'_0, g'_0)$ is an isometry, then there exists a unique diffeomorphism $\hat\phi: \mathcal{M} \rightarrow \mathcal{M}'$ such that
        \begin{equation*}
            \hat \phi^* g' = g, \quad \hat \phi|_{\mathcal{M}_0} = \phi, \quad \hat \phi_* \partial_{\mathfrak{t}} = \partial_{\mathfrak{t}'}, \quad \mathfrak{t}'\comp \phi = \mathfrak{t}\,.
        \end{equation*}
    \end{theorem}

    These two results indicate the canonicality of singular Ricci flows, and thus, leads to the existence and uniqueness of equivariant singular Ricci flows.

    \begin{theorem}{\label{thm equivariance singular ricci flow}}
        If $(M, g)$ is a $G$-invariant Riemannian three-manifold, then the unique singular Ricci flow $(\mathcal{M}, \mathfrak{t}, \partial_{\mathfrak{t}}, g)$ with initial data $(M, g)$ is also a $G$-invariant singular Ricci flow.
    \end{theorem}

    \begin{proof}
        We denote the group action on $M$ by $\rho: G \acts M$. For any $h \in G$, by the uniqueness theorem \ref{thm:unique}, there is a unique isometry $\hat\rho_h:\mathcal{M}\to\mathcal{M}$ with the property that $\hat\rho_h = \rho_h $ on $\mathcal{M}_0 = M$. It suffices to show that $\set{\hat\rho_h}_{h \in G}$ is also a group action. Indeed, for $h_1, h_2 \in G$, $\rho_{h^{-1}_1h_2} = \rho_{h^{-1}_1}\comp\rho_{h_2}$, and thus, by the uniqueness theorem again, we have
        \begin{equation*}
            \hat\rho_{h^{-1}_1h_2} = \hat\rho_{h^{-1}_1}\comp\hat\rho_{h_2}\,.
        \end{equation*}
        
        In conclusion, $\hat\rho: G \acts \mathcal{M}$ defined above is a group action, and  is a singular Ricci flow.
    \end{proof}

\section{Equivariant rounding process}

    In this section, we shall modify the rounding process (Section 5) in \cite{bamler_ricci_2019} such that perturbed metrics on Ricci flow spacetimes are equivariant so long as the initial data are.

    Let $M$ be a smooth manifold of dimension $n = 3$ or $4$, $G$ be a compact Lie group and $\rho:G \curvearrowright M$ be a diffeomorphic action. For simplicity, we will omit the notation $\rho$ if it is clear and use $h$ instead of $\rho_h$ to denote the action for $h \in G$.

    \begin{definition}[Equivariant spherial structure]
        A \textbf{$G$-invariant spherical structure} $\mathcal{S}$ on a $G$-invariant subset $U \subset M$ is a smooth $S^2$-fiber bundle structure an open dense $G$-invariant subset $U' \subset U$ with the following properties.
        \begin{enumerate}[label = (\arabic*)]
            \item The bundles structure is equipped with a smooth fiberwise metric of constant curvature $1$;
            \item For every $x \in U$, there exists a neighborhood $V \subset U$ of $x$ and an $O(3)$-action $\xi: O(3) \acts V$ preserving all $S^2$-fibers in $V\cap U'$ as isometric and effective fiberwise actions;
            \item All orbits on $U \setminus U'$, called singular (spherical) fiber, are diffeomorphic to either a point or $\mathbb{R}P^2$;
            \item Moreover, for any fiber $F$ and any $h \in G$, $h.F$ is also a fiber of $\mathcal{S}$.
        \end{enumerate}
        We call $U = \mathrm{domain}(\mathcal{S})$ the domain of $\mathcal{S}$, and each $S^2$-fiber a regular fiber.   
    \end{definition}

    \begin{definition}[Compatible metrics, {\cite[Definition 5.4]{bamler_ricci_2019}}]
        On a $G$-invariant subset $U \subset M$, a $G$-invariant metric $g$ is compatible with a $G$-invariant spherical structure $\mathcal{S}$, if near every point in $U$, there is a local $O(3)$-action compatible with $\mathcal{S}$ is isometric with respect to $g'_t$.
    \end{definition}
    
    \begin{definition}[Equivariant $\mathcal{R}$-structure]
        A \textbf{$G$-invariant $\mathcal{R}$-structure} on a $G$-invariant singular Ricci flow $(\mathcal{M}, \mathfrak{t}, \partial_{\mathfrak{t}}, g)$ is a tuple $(g', \partial'_{\mathfrak{t}}, U_{S_2}, U_{S_3}, \mathcal{S})$ with the following properties.
        \begin{enumerate}[label = (\arabic*)]
            \item $\mathcal{S}$ is a $G$-invariant spherical structure, each fiber of which has a constant $\mathfrak{t}$;
            \item $g'$ is a $G$-invariant metric on $\ker \mathrm{d}\mathfrak{t}$ with $g'_t$ compatible with $\mathcal{S}$, which is obtained from $g$ by a rounding process near singularities;
            \item $\partial'_{\mathfrak{t}}$ is a $G$-invariant vector field on $\mathcal{M}$ with $\partial'_{\mathfrak{t}} \mathfrak{t} = 1$ and preserving $\mathcal{S}$;
            \item $U_{S_2} = \mathrm{domain}(\mathcal{S})$ is a $G$-invariant open subset;
            \item $U_{S_3}$ is a $G$-invariant open subset such that for any $t$, $U_{S_3} \cap \mathcal{M}_t$ is a union of compact components with constant curvature with respect to $g'_t$;
            \item $U_{S_3} \setminus U_{S_2}$ is open; 
            \item $U_{S_3}$ is invariant under the forward flow of $\partial'_{\mathfrak{t}}$ and each component at $t\geq 0$ evolves homothetically;
        \end{enumerate}
    \end{definition}

    With these definitions, our main goal is to show the existence of a transverse continuous family of $G$-invariant $\mathcal{R}$-structures (See \cite[Definition~5.11]{bamler_ricci_2019}). These is an equivariant version of \cite[Theorem~5.12]{bamler_ricci_2019}.
  
    \begin{theorem}\label{thm:equiv_round}
        For any $\delta > 0$, there exists a constant $C = C(\delta) < \infty$ and a continuous, decreasing function $r_{\mathrm{rot}, \delta}: \mathbb{R}_+ \times [0, \infty) \rightarrow \mathbb{R}_+$ such that the following holds.

        Any continuous family $(\mathcal{M}^s)_{s \in X}$ of $G$-invariant singular Ricci flows admits a transversely continuous family of $G$-invariant $\mathcal{R}$-structures $(\mathcal{R}^s = (g'^{,s}, \partial'^{,s}_{\mathfrak{t}}, U^s_{S_2}, U^s_{S_3}, \mathcal{S}^s))_{s \in X}$ such that for any $s \in X$:
        \begin{enumerate}[label = (\alph*)]
            \item $\mathcal{R}^s$ is supported on 
                \[
                    \set{x \in \mathcal{M}^s: \rho_{g'^{,s}}(x) < r_{\mathrm{rot},\delta}(r_{\mathrm{initial}}(\mathcal{M}^s_0, g^s_0), \mathfrak{t}(x))}\,.
                \]
            \item $g'^{,s} = g^s$ and $\partial'^{,s}_{\mathfrak{t}} = \partial^s_{\mathfrak{t}}$ on
                \[
                    \set{x \in \mathcal{M}^s: \rho_{g'^{,s}}(x) > C r_{\mathrm{rot},\delta}(r_{\mathrm{initial}}(\mathcal{M}^s_0, g^s_0), \mathfrak{t}(x))}\supset \mathcal{M}^s_0\,.
                \]
            \item For $m_1, m_2 = 0, \cdots, [\delta^{-1}]$ we have:
                \[
                    |\nabla^{m_1}\delta^{m_2}_\mathfrak{t}(g'^{,s} - g^s)| \leq \delta \rho^{-m_1-2m_2}, \quad |\nabla^{m_1}\partial^{m_2}_{\mathfrak{t}}(\partial'^{,s}_{\mathfrak{t}}-\partial^s_{\mathfrak{t}})|\leq \delta \rho^{1-m_1 -2m_2}\,.
                \]
            \item If $(\mathcal{M}^s_0, g^s_0)$ is homothetic to a quotient of the round sphere or the round cylinder, then $g'^{,s} = g^s$ and $\partial^{',s}_{\mathfrak{t}}=\partial^s_{\mathfrak{t}}$ on all of $\mathcal{M}^s$.
            \item $r_{\mathrm{rot},\delta}(a r_0, a^2 t) = a\cdot r_{\mathrm{rot}, \delta}(r_0, t)$ for all $a, r_0 > 0$ and $t \geq 0$.
        \end{enumerate}
    \end{theorem}

    \begin{proof}
        We shall follow closely the proof of \cite[Theorem~5.12]{bamler_ricci_2019}. Though not indicated explicitly, various geometric objects constructed \textit{loc. cit.} are canonical, and therefore lead to $G$-invariance of $\mathcal{R}$-structures. Here, we shall sketch its proof to verify the equivariance.\\

        \paragraph{\textbf{Step 1}} Construction of an equivariant cutoff function near almost extinct components.

            The first step is to construct an initial equivariant cutoff function $\eta_1: \bigcup_{s \in X} \mathcal{M}^s \rightarrow [0, 1]$ to indicate the region where the spherical structures will be defined. More precisely, we verify that in {\cite[Lemma~5.15]{bamler_ricci_2019}}, under the $G$-invariant assumptions, the cutoff function $\eta_1$ constructed in the proof is $G$-invariant.
                
            By definition, functions $\rho$, $r_{\mathrm{can}, \varepsilon}$ and $\hat \rho$ are $G$-invariant. Hence, in the proof of Lemma 5.15 in \cite{bamler_ricci_2019}, functions
            \begin{align}
                \eta^*_1(x) &:= \nu(a \cdot \frac{r_{\mathrm{can}, \varepsilon}(x)}{\hat \rho(x)})\,,\\
                \eta^{**}_1(x) &:= \nu(A \cdot \frac{\hat \rho(x)}{\mathcal{R}(\hat{\mathcal{C}})})\,,
            \end{align}
            are $G$-invariant. Note that $\nu: \mathbb{R} \rightarrow \mathbb{R}$ is a standard cutoff function.

            Since the time vector $\partial_{\mathfrak{t}}$ is $G$-invariant, for any $h \in G$ and $x, x' \in \mathcal{M}^s_t$, if $x' = h(x)$, then $x'(t') = h(x(t'))$ as long as $x$ survives until time $t'$. Therefore, $\eta'_1$ defined by
                \[
                    \eta'_1(x) := \int^t_{t^*_{\mathcal{C}}} \frac{\eta^*_1(x(t'))\eta^{**}_1(x(t'))}{\hat \rho^2(x(t'))}\mathrm{d}t'\,,
                \]
            is $G$-invariant and so is $\eta_1(x) = \nu(A \cdot \eta'_1(x))$.
            
            {\ }

        \paragraph{\textbf{Step 2}} Equivaraint modification in regions that are geometrically close to Bryant solitons.

            We check that in \cite[Lemma~5.20]{bamler_ricci_2019}, $(g'^{,s}_2)_{s \in X}$, $\eta_2$ and $(\mathcal{S}^s_2)_{s\in X}$ constructed in the proof are also $G$-invariant.
            
            \begin{enumerate}[label = (\roman*)]
                \item Apparently, $\eta_2(x):= \nu(2\eta_1(x))$ is $G$-invariant.

                \item Let's show that $(g'^{,s}_2)_{s \in X}$ are $G$-invariant.

                    As $\eta_1$, $\rho$, $r_{\mathrm{can}, \varepsilon}$ and $R$ are $G$-invariant, the set $E^s$ (the set of tips) therein is a $G$-invariance subset by definition.

                    For any $x \in E^s \cap \mathcal{M}^s_t$ and any $h \in G$, $x' = h.x$ is also in $E^s \cap \mathcal{M}^s_t$ and $(B(x, 10 D_0 \rho(x)), g^{s}_{t})$ is isometric to $(B(x', 10 D_0 \rho(x')), g^{s}_{t})$ via $h$. Note that $g''$ is constructed as the average of $g$  which is canonical and independent of the choice of $\varphi$. Hence, $g''$ and $h^*(g'')$ are the same when they are restricted to $B(x, 10 D_0 \rho(x))$.

                    By the arbitrariness of $x$ and $h$, $g''$ is $G$-invariant. Since the distance function $d$ is $G$-invariant, by definition, $g'''$ defined as a linear interpolation between $g$ and $g''$ w.r.t. distance function $d$ is $G$-invariant.

                    Finally, $g'_2$ defined a linear interpolation between $g$ and $g'''$ via a $G$-invariant cutoff function is again $G$-invariant.

                \item It follows from the construction of $g''$ that $(\mathcal{S}'^{,s})_{s \in X}$ are all $G$-invariant. Therefore, so are the restrictions $(\mathcal{S}_2^s)_{s \in X}$ to a $G$-invariant subset.
            \end{enumerate}{\ }

        \paragraph{\textbf{Step 3}} Equivariant modification in cylindrical regions.

            We show that in \cite[Lemma~5.24]{bamler_ricci_2019}, $(g'^{,s}_3)_{s \in X}$, $\eta_3$ and $(\mathcal{S}^s_3)_{s\in X}$ in the proof are $G$-invariance.

            \begin{enumerate}[label = (\roman*)]
                \item $\eta_3$ is $G$-invariant because $\eta_3(x) = \nu(2 \eta_2(x))$.
                \item To prove the $G$-invariance of $(g'^{,s}_3)_{s \in X}$, by definition, it suffices to show that $g''^{,s}$ therein is $G$-invariant.

                    Firstly, for any $x \in \mathcal{M}^s_t$ satisfying Assertion (c2) or (c3) of \cite[Lemma~5.20]{bamler_ricci_2019}, the corresponding CMC sphere or CMC-projective space $\Sigma_x$ w.r.t. $g'^{,s}_{2, t}$ is unique. Therefore, for any $h \in G$, $\Sigma_{h.x} = h.\Sigma_x$. Moreover, $g'_x = h^*(g'_{h.x})$ on $\Sigma_x$ by the canonicality of $\operatorname{RD}^2$, so the spherical structures $(\mathcal{S}'^{,s})_{s \in X}$ are $G$-invariant.

                    Next, the uniqueness of $Z_{\Sigma_x, N_x}$ implies that $h_*(Z_{\Sigma_x, N_x}) = Z_{\Sigma_{h(x)}, N_{h(x)}}$. By the arbitrariness of $x$ and $h$, it follows immediately that $g''^{,s}$ is $G$-invariant.

                \item The conclusion that $\mathcal{S}^s_3$ is $G$-invariant follows from the fact that $\mathcal{S}'^{,s}$ is $G$-invariant, which was proved in Part (ii).
            \end{enumerate}{\ }

        \paragraph{\textbf{Step 4}}{Equivariant modification of the time vector field}.

            We show that in \cite[Lemma~5.27]{bamler_ricci_2019}, $(g'^{,s}_4)_{s \in X}$, $\eta_4$, $(\mathcal{S}^s_4)_{s\in X}$ and $\partial'^{,s}_{\mathfrak{t}, 4}$ constructed in the proof are $G$-invariant.

            We follow the proof of the lemma to verify $G$-invariance.
            \begin{enumerate}[label = (\roman*)]
                \item We first show that the vector field $Z^s$ constructed in the proof of \cite[Lemma~5.27]{bamler_ricci_2019} is $G$-invariant. Note that its domain $U^s = \set{\eta_3 < 1} \cap \set{\rho_{g'_3} < \frac{1}{2}\alpha r_{\mathrm{can},\varepsilon}} \cap \mathcal{M}^s$ is $G$-invariant.

                Fix a spherical fiber $\mathcal{O}$ and $h \in G$, and in the proof, there exist admissible vector fields $Z'_{\mathcal{O}}$ and $Z'_{h.\mathcal{O}}$ with minimal $L^2$-norms. It suffices to show that $h_*(Z'_{\mathcal{O}}) = Z'_{h.\mathcal{O}}$.

                Recall that $h$ is an isomorphism, so $h_*(Z'_{\mathcal{O}})$ is an admissible vector field with the minimal $L^2$-norms as well. 

                Suppose by contradiction that $h_*(Z'_{\mathcal{O}}) \neq Z'_{h.\mathcal{O}}$, and we can construct a new vector field along $h.\mathcal{O}$:
                \begin{equation}
                    Z'' := \frac{h_*(Z'_{\mathcal{O}}) + Z'_{h.\mathcal{O}}}{2}\,.
                \end{equation}

                On the one hand, $Z''$ is admissible since the space of admissible vector fields along $h.\mathcal{O}$ is affine. On the other hand, by Cauchy-Schwarz inequality, we have
                \begin{equation}
                    \|Z''\|^2_{L^2} < (\|Z'_{h.\mathcal{O}}\|^2_{L^2} + \|h_*(Z'_{\mathcal{O}}\|^2_{L^2})/2\,,
                \end{equation}
                giving a contradiction that $h_*(Z'_{\mathcal{O}})$ and $Z'_{h.\mathcal{O}}$ have the minimal $L^2$-norm.
                
                Therefore, by the arbitrariness of $h$ and $\mathcal{O}$, $Z^s$ is $G$-invariant.            

                \item It follows immediately from Part (i) that $\partial'^{,s}_{\mathfrak{t}, 4} = \partial^s_{\mathfrak{t}} + Z^s$ is $G$-invariant.

                \item $\mathcal{S}^s_4$ is the restriction of $\mathcal{S}^s_3$ to a $G$-invariant domain $U' \cap \mathcal{M}^s$, so it is $G$-invariant.

                \item $\eta_4(x) := \nu(2 \eta_3(x))$ is apparently $G$-invariant.

                \item $(g'^{,4}_s)_{s \in X} = (g'^{,3}_s)_{s \in X}$ are $G$-invariant.
            \end{enumerate}{\ }

        \paragraph{\textbf{Step 5}} {Equivariant extension of the structure until the metric is almost round.}
        
            We show that in \cite[Lemma~5.29]{bamler_ricci_2019}, $(g'^{,s}_5)_{s \in X}$, $\eta_5$, $(\mathcal{S}^s_5)_{s\in X}$ and $\partial'^{,s}_{\mathfrak{t}, 5}$ constructed are $G$-invariant.

            \begin{enumerate}
                \item $\partial'^{,s}_{\mathfrak{t}, 5} = \partial'^{,s}_{\mathfrak{t}, 4}$ is obviously $G$-invariant by the previous lemma.
                \item Since $\eta_4$ is $G$-invariant, $g'_5 = g'_4$, $\eta_5 = 0$ and $\mathcal{S}_5 = \mathcal{S}_4$ on $\set{\eta_4} = 0$, they are all $G$-invariant there by the previous lemma.
                \item Let's consider the cutoff function $\eta_5$, $g'^{,5}$ and $\mathcal{S}_5$ on $\set{\eta_4 > 0}$.

                    Note that $\eta_4$ is $G$-invariant and so is the open subset $\set{\eta_4 > 0}$. Furthermore, for any connected component $\mathcal{C} \subset \mathcal{M}^t_s \cap \set{\eta_4 > 0}$ and $h \in G$, we have $t^{\min}_{\mathcal{C}} = t^{\min}_{h.\mathcal{C}}$, $t^{\max}_{\mathcal{C}} = t^{\max}_{h.\mathcal{C}}$ and thus, $h.U_{\mathcal{C}} = U_{h.\mathcal{C}}$.

                    The vector field $\tilde{\partial}^s_t = \hat \rho^2_g \cdot \partial'^{,s}_{\mathfrak{t}, 4}$ is $G$-invariant and so are its generating flow $\Phi$, $\tilde g_t$, $\tilde{g}_{4, t}$, $\tilde{\eta}_4$ and the shifted vector field $Z_t$ as defined therein.

                    As the volume normalized Ricci flow equation has unique solution, $\tilde g_{5, t} = \operatorname{RF}_{\tilde{\eta}_{4(t)A}}\tilde{g}_{4, t - \tilde{\eta}_4(t)A}$ and $\tilde{\eta}_5(t)$ are $G$-invariant. By definitions that $\tilde{g}_{5, t} = \Phi^*_t(\tilde{\rho}^{-2}_g g^{',s}_5)$ and $\tilde{\eta}_5(t) = \eta_5(\Phi_t(\mathcal{C}))$, we have $\eta_5$ and $g'^{,5}$ are $G$-invariant on $\set{\eta_4 > 0}$.

                    Since $\mathcal{S}^s_5$ is the push-forward of $\mathcal{S}^s_4|_{\Phi_{t - \tilde{\eta}_4(t)A}(\mathcal{C})}$ onto $\Phi_t(\mathcal{C})$ via $\Phi_t \circ \Phi^{-1}_{t - \tilde{\eta}_4(t)A}$, it is $G$-invariant as well on $\set{\eta_4 > 0}$.
            \end{enumerate}{\ }

        \paragraph{\textbf{Step 6}}{Modification of almost round components.}

            \begin{enumerate}[label = (\roman*)]
                \item Following the proof of \cite[Lemma~5.12]{bamler_ricci_2019}, we first consider $g'^{,s}$ and $\partial'^{,s}_{\mathfrak{t}}$.

                    Set $g''^{,s}_t|_{\mathcal{C}} = \operatorname{RD}^3(g'^{,s}_{5, t}|_{\mathcal{C}})$ on $\mathcal{M}^s \cap \set{\eta_5 > 0}$, and it is easy to see that it is a $G$-invariant metric. Hence, $g'^{,s}$ are $G$-invariant metrics by definition.

                    Then, similar to what we have shown in Part (i) of Step 4, the constructed vector fields $Z^s$ with minimal $L^2$-norm when restricted to time slices are $G$-invariant. So are $\partial''^{,s}_{\mathfrak{t}} = \partial''^{,s}_{\mathfrak{t},5} + Z''$ on $\set{\eta_5 > 0}$. By definition, $\partial'^{,5}_{\mathfrak{t}}$ are $G$-invariant.

                \item Now, let's verify that the construction of spherical structures is $G$-invariant.

                    By definition, $U^s_{S_3} = \set{\hat \rho_g < c \alpha r_{\mathrm{can}, \varepsilon}} \cap \set{\frac{1}{2}< \eta_5} \cap \mathcal{M}^s$ is apparently $G$-invariant.

                    Given the $G$-invariant $\mathcal{S}^s_5$, the restricting spherical structure $\mathcal{S}^s_6$ is also $G$-invariant since
                    \begin{equation}
                        \mathrm{domain}(\mathcal{S}^s_6) = (\mathrm{domain}(\mathcal{S}^s_5)\cap\set{\eta_5 < \frac{1}{4}})\cup(\set{\hat{\rho}_g < c \alpha r_{\mathrm{can},\varepsilon}}\cap\set{0 < \eta_5 < 1}\cap \mathcal{M}^s)\,.
                    \end{equation}

                    Since $U_{S_2}$ is the union of $\mathrm{domain}(\mathcal{S}^s_6)$ with all components $\mathcal{C} \subset \mathcal{M}^s_t \cap U^s_{S_3}$ intersecting $\mathrm{domain}(\mathcal{S}^s_6)$, it is a $G$-invariant subset as well.

                    Finally, $\mathcal{S}^s$ is extended from $\mathcal{S}^s_5$ on $U^s_{S_2}$ by the flow of $\partial'^{,s}_{\mathfrak{t}}$, so is a $G$-invariant spherical structure.
            \end{enumerate}
    \end{proof}

\section{Equivariant partial homotopies}

    In this section, we shall introduce the equivariant version of partial homotopies in \cite{bamler_ricci_2019}. Let us first fix a simplicial complex $K$ and consider a fiber bundle $E\to K$ whose fibers are smooth compact Riemannian 3-manifolds, which we view as a continuous family of $G$-invariant Riemannian manifolds $(M^s, g^s)_{s\in K}$. We then fix a sub-complex $L\subset K$ such that $(M^s, g^s)_{s\in L}$ contains all the CC-metrics. Let $(\mathcal{M}^s)_{s\in K}$ be a continuous family of singular Ricci flows whose time-0-slices $(\mathcal{M}_0^s, g_0^s)_{s\in K}$ is isometric to $(M^s, g^s)_{s\in K}$. Moreover, we fix  a transversely continuous family of $G$-invariant $\mathcal{R}$-structures
        \[\mathcal{R}^s = (g'^{,s}, \partial'^{,s}_t, U_{S_2}^s, U_{S_3}^s, \mathcal{S}^s)_{s\in K}\]
        for each $\mathcal{M}^s$. Fix some $T\geq 0$.
        
    \begin{definition}[Equivariant Partial homotopy]{\label{equivariant partial homotopy}}
     With the settings above, a partial homotopy $\{(Z^{\sigma}, (g_{s,t}^{\sigma})_{s\in\sigma, t\in [0,1]}, (\psi_s^{\sigma})_{s\in\sigma})\}_{\sigma\subset K}$ (\cite[Definition~7.4]{bamler_ricci_2019}) at time $T$ relative to $L$ is called $G$-invariant if
         \begin{itemize}
        \item[(i)] For all $s\in\sigma\subset K$, the images $\psi_s^{\sigma}(Z^{\sigma})$ is a $G$-invariant subset of $M_T^s$. 
        \item[(ii)] For all $s\in\sigma\subset K$ and $t\in [0,1]$, $(\psi_s^{\sigma})_*g^{\sigma}_{s,t}$ is a $G$-invariant metric on the image $\psi_s^{\sigma}(Z^{\sigma})$.
        \end{itemize}
    \end{definition}

    \begin{definition}{\label{induced action}}
        Consider an embedding $\psi:Z \to \mathcal{M}_T^s$ whose image $\psi(Z)$ is a G-invariant subset of $\mathcal{M}_T^s$. We define the induced group action $G\acts Z$ by
        \[ h.z := \psi^{-1}(h. \psi(z)) \]
        for all $h\in G$ and $z\in Z$. In particular, the induced group action commutes with the embedding $\psi$.
    \end{definition}
    \begin{remark}\label{G-invariance of induced action}
    $\psi_*g$ is a $G$-invariant metric on the image $\psi(Z)$ if and only if $(Z, g)$ is $G$-invariant with respect to the induced group action in Definition \ref{induced action}.
    \end{remark}

    In the sequel of this section, we shall adapt the results from Section 7 of \cite{bamler_ricci_2019} to the equivariant setting.

    \begin{proposition}[Equivariant modification of moving a partial homotopy backward in time]{\label{equivariant modification of moving a partial homotopy backwards in time}} 
        In \cite[Proposition~7.6]{bamler_ricci_2019}, Suppose that the preassigned partial homotopy 
            $$\{(Z^{\sigma}, (g^{\sigma}_{s,t})_{s\in\sigma, t\in [0,1]}, (\psi_s^{\sigma})\}_{\sigma\subset K}$$ 
        at time $T$ relative to $L$ is $G$-invariant in the sense of Definition \ref{equivariant partial homotopy}. Then, the resulting partial homotopy $\{(Z^{\sigma}, (\bar{g}^{\sigma}_{s,t})_{s\in\sigma, t\in [0,1]}, (\bar{\psi}_s^{\sigma})\}_{\sigma\subset K}$ at time $T'$ relative to $L$ given in the conclusion therein is also $G$-invariant. 
    \end{proposition}

    \begin{proof}
        The map $\bar{\psi}_s^{\sigma}$ is defined by
            \[ \bar{\psi}_s^{\sigma} := (\psi^{\sigma}_t)_{T'}\]
        with $(\psi^{\sigma}_t)_{t'}(z) := (\psi^{\sigma}_s(z))^{\partial'^{,s}_{\mathfrak{t}}}(t')$. Since the vector field $\partial'^{,s}_{\mathfrak{t}}$ is $G$-invariant, it follows that if $U\subset Z^{\sigma}, \sigma\subset K$, and $\psi_s^{\sigma}(U)$ is a $G$-invariant subset, then  $\bar{\psi}_s^{\sigma}(U)$ is also a G-invariant subset.  Next, by the construction of the new metric 
            \begin{align*}
                \bar{g}^{\sigma}_{s,t} :=
                \begin{cases}
                    (\psi^{\sigma}_s)^*_{T'+2t(T-T')}g_{T'+2t(T-T')} \quad &\text{if}\quad t\in [0,\frac{1}{2}]\\
                    g^{\sigma}_{s, 2t-1} \quad &\text{if}\quad t\in [\frac{1}{2}, 1],
                \end{cases}
            \end{align*}
            it follows from the G-invariance of $\psi^{\sigma}_s$, $g_{T'+2t(T-T')}$, $g^{\sigma}_{s, 2t-1}$ and the definition of $\bar{\psi}_s^{\sigma}$ that the image $\bar{\psi}_s^{\sigma}(Z^{\sigma})$ is a $G$-invariant subset and $(\bar{\psi}_s^{\sigma})_*\bar{g}^{\sigma}_{s,t}$ is a G-invariant metric on the image $\bar{\psi}_s^{\sigma}(Z^{\sigma})$.
    \end{proof}

    It is straightforward to see that the proof of \cite[Proposition~7.8]{bamler_ricci_2019} carries over to the equivariant setting directly. Thus we have

    \begin{proposition}[Equivariant modification of passing to a simplicial refinement]
        Let $(K', L')$ be simplicial refinement of $(K, L)$. Suppose that the preassigned partial homotopy 
    $$\{(Z^{\sigma}, (g^{\sigma}_{s,t})_{s\in\sigma, t\in [0,1]}, (\psi_s^{\sigma})\}_{\sigma\subset K}$$ at time $T$ relative to $L$ in \cite[Proposition~7.8]{bamler_ricci_2019} is $G$-invariant. Then, the resulting partial homotopy
    $$\{(\bar{Z}^{\sigma}, (\bar{g}^{\sigma}_{s,t})_{s\in\sigma, t\in [0,1]}, (\bar{\psi}_s^{\sigma})\}_{\sigma\subset K'}$$
    at time $T$ relative to $L'$
    \end{proposition}

    \begin{proposition}[Equivariant modification of enlarging a partial homotopy]{\label{equivariant modification of enlarging a partial homotopy}} Suppose that the preassigned partial homotopy 
    $$\{(Z^{\sigma}, (g^{\sigma}_{s,t})_{s\in\sigma, t\in [0,1]}, (\psi_s^{\sigma})\}_{\sigma\subset K}$$ at time $T$ relative to $L$ in \cite[Proposition~7.9]{bamler_ricci_2019} is $G$-invariant. Fix $\sigma\subset K$, we assume in addition that the embeddings $\hat{\psi}_s^{\sigma}(\hat{Z}_s^{\sigma})$ are $G$-invariant subsets of $\mathcal{M}_T^s$.
    Then, the resulting partial homotopy
    \[\{(Z^{\sigma'}, (g^{\sigma'}_{s,t})_{s\in\sigma', t\in [0,1]}, (\psi_s^{\sigma'})\}_{\sigma'\subset K, \sigma'\neq\sigma}\cup \{(\hat{Z}^{\sigma}, (\hat{g}^{\sigma}_{s,t})_{s\in\sigma, t\in [0,1]}, (\hat{\psi}_s^{\sigma})\}\]
    at time $T$ relative to $L$ constructed therein is $G$-invariant PSC conformal in the sense as in Appendix \ref{app:1}.
    \end{proposition}

    \begin{proof}
        We follow closely from the original proof. Firstly, by Definition \ref{induced action} the induced action on $\hat{Z}^{\sigma}$ is compatible with the induced action on $Z^{\sigma}$ via the embedding $\iota^{\sigma}:Z^{\sigma}\to\hat{Z}^{\sigma}$. Note that by Remark \ref{G-invariance of induced action}, $g_{s,t}^{\tau}$ is a $G$-invariant metric on $Z^{\tau}$ for all $s\in\tau\subset\partial\sigma$.   Hence the continuous family of metrics 
            \[k_{s,t} = ((\psi_s^{\tau})^{-1}\circ\hat{\psi}_s^{\sigma})^*g_{s,t}^{\tau}\]
        constructed on $\hat{Z}^{\sigma}$ for all $s\in\tau\subset\partial\sigma$, $t\in [0,1]$ in the original proof is $G$-invariant on $\hat{Z}^{\sigma}$, as the group action maps commute with $\psi_s^{\tau}$ and $\hat{\psi}_s^{\sigma}$ by compatibility. We next check that the metrics $(\hat{g}^{\sigma}_{s,t})_{s\in\sigma, t\in [0,1]}$ constructed in the original proof are $G$-invariant on $\hat{Z}^{\sigma}$ case-by-case.\\
     
        \begin{itemize}
        \item[\textsl{Case 1a}:]     
            Observe that $k_{s,t}$ is $G$-invariant for all $s\in\partial\sigma$, $t\in[0,1]$ and that $(\hat{\psi}_s^{\sigma})^*g'^{,s}_t$ is also  $G$-invariant on $\hat{Z}^{\sigma}$ by Remark \ref{G-invariance of induced action} for $s\in\sigma$. Hence the family of metrics $(\hat{g}^{\sigma}_{s,t})_{s\in\sigma, t\in [0,1]}$ constructed in the original proof using \cite[Proposition~6.15]{bamler_ricci_2019} is $G$-invariant by Proposition \ref{prop:equiv-symm-met}.

        \item[\textsl{Case 1b}:] 
            Since $\mathcal{S}^s$ is $G$-invariant for any $s\in A$, the spherical structure $\mathcal{S}^{',s}$ obtained by pull back is also $G$-invariant with respect to the action induced by the embedding $\hat{\psi}^{\sigma}|_{Y}$. Moreover, $k_{s,t}$ is $G$-invariant for all $s\in\partial\sigma$, $t\in[0,1]$. Thus Proposition A.2 implies that the family of metrics $(\hat{g}^{\sigma}_{s,t})_{s\in\sigma, t\in [0,1]}$ constructed in the original proof using \cite[Proposition~6.18]{bamler_ricci_2019} is $G$-invariant by Proposition A.2.

        \item[\textsl{Case 2}:] 
            The $G$-invariance of $(\hat{g}^{\sigma}_{s,t})_{s\in\sigma, t\in [0,1]}$ follows from the $G$-invariance of $(k_{s,t})_{s\in\sigma, t\in [0,1]}$.\\
        \end{itemize}
        Moreover, it follows from Lemma \ref{lem equiv PSC conformal} that the new partial homotopy is $G$-invariant PSC-conformal.
     
    \end{proof}

    \begin{proposition}[Equivariant modification of removing disks from a partial homotopy]{\label{equivariant modification of removing a disk from a partial homotopy}}
        Suppose that in the assumptions of \cite[Proposition~7.21]{bamler_ricci_2019}:     \begin{itemize}
        \item[(1)] The preassigned partial homotopy $\{(Z^{\sigma}, (g^{\sigma}_{s,t})_{s\in\sigma, t\in [0,1]}, (\psi_s^{\sigma}))\}_{\sigma\subset K}$ at time $T$ relative to $L$  is  $G$-invariant in the sense of Definition \ref{equivariant partial homotopy}.       
        \item[(2)]  The embeddings $\{(\nu_{s,j}: D^3(1)=: D^3\to \mathcal{M}_T^s)_{s\in\sigma_0}\}_{1\leq j\leq m}$ are $G$-invariant subsets of $\mathcal{M}_T^s$. 
        \end{itemize}

    Then, the resulting partial homotopy
    \[\{(\tilde{Z}^{\sigma}, (\tilde{g}^{\sigma}_{s,t})_{s\in\sigma, t\in [0,1]}, (\tilde{\psi}_s^{\sigma}))\}_{\sigma\subset K}\]
    at time $T$ relative to $L$ constructed in \cite[Proposition~7.21]{bamler_ricci_2019} is $G$-invariant.
    \end{proposition}

    \begin{proof}
        We follow closely from the original proof. Since the modification is constructed locally around each disk, we may only consider an embedding disk $\set{\nu_s:D^3(1) \rightarrow \mathcal{M}^s_T}_{s\in \sigma_0}$ and the subgroup $\set{g \in G| g.(\nu_s(D^3)) = \nu_s(D^3)}$, still denoted by $G$ for simplicity. The modification of other disks could be either induced by a group action on $\nu_s(D^3)$ or constructed by induction.

        Firstly the extension $(\bar{\nu}_s: D^3(2)\to \mathcal{M}^s)_{s\in U}$ in the original proof can be constructed in a way that the embedded image of $\bar{\nu}_s$ consists of a union of fibers of $U_{S_2}^s$. In this way the embedded image $\bar{\nu}_s(D^3(2))$ is still a $G$-invariant subset of $\mathcal{M}_T^s$. Following the original proof, we replace $(g^{\sigma_0}_{s,t})_{s\in\sigma, t\in [0,1]}, (\psi_s^{\sigma_0})_{s\in\sigma_0}$ by $(\bar{g}^{\sigma_0}_{s,t})_{s\in\sigma, t\in [0,1]}, (\bar{\psi}_s^{\sigma_0})_{s\in\sigma_0}$ so that after the replacement the embedding $\mu:= (\psi_s^{\sigma_0})^{-1}\circ\bar{\nu}_s$ is constant in $s$. We now make the following claim:
        \begin{claim}{\label{claim G-invariant mu}}
            The embedding $\mu: D^3(2)\to Z^{\sigma_0}$ is $G$-invariant with respect to the group actions induced by $\psi_s^{\sigma_0}: Z^{\sigma_0}\to \mathcal{M}_T^s$ for all $s\in\sigma_0$. 
        \end{claim}
        \begin{proof}[Proof of Claim \ref{claim G-invariant mu} ]
            To see why this is true, we fix $s\in\sigma_0$, a group element $h\in G$ and an arbitrary pair of points $x, y\in D^3(2)$ such that $\bar{\nu}_s(y) = h.\bar{\nu}_s(x)$.  Then with respect to the group action  induced by $\psi_s^{\sigma_0}: Z^{\sigma_0}\to \mathcal{M}_T^s$, we have
            \begin{align*}
                h.\mu(x) &= (\psi_s^{\sigma_0})^{-1}(h.(\psi_s^{\sigma_0} (\mu(x))))\\
                &=  (\psi_s^{\sigma_0})^{-1}(h.\bar{\nu}_s(x))\\
                &=  (\psi_s^{\sigma_0})^{-1}(\bar{\nu}_s(y))\\
                &= \mu(y).
            \end{align*}
            Since $\nu_s(D^3(1))$ is a $G$-invariant subsets of $\mathcal{M}_T^s$, we can verify the assertions for $\mu(D^3(1))$ similary. This implies the claim.
         \end{proof}
     
        In the sequel we also write $\nu_s$ instead of $\bar{\nu}_s$ as in the original proof. We next verify that the new partial homotopy $\{(\tilde{Z}^{\sigma}, (\tilde{g}^{\sigma}_{s,t})_{s\in\sigma, t\in [0,1]}, (\tilde{\psi}_s^{\sigma}))\}_{\sigma\subset K}$ constructed in the original proof is $G$-invariant case by case.\\

        \paragraph{\textbf{Case 1}: $\sigma_0\subsetneq\sigma_1$ for some simplex $\sigma_1\subset K$.} \

            From the original proof, the new partial homotopy $\{(\tilde{Z}^{\sigma}, (\tilde{g}^{\sigma}_{s,t})_{s\in\sigma, t\in [0,1]}, (\tilde{\psi}_s^{\sigma})\}_{\sigma\subset K}$ is given by
                \[\tilde{Z}^{\sigma_0} := Z^{\sigma_0}\setminus\mu(\text{Int}D^3(1)),\quad \tilde{\psi}_s^{\sigma_0}:= \psi_s^{\sigma_0}\big|_{\tilde{Z}^{\sigma_0}},\quad  \tilde{g}^{\sigma_0}_{s,t} := g^{\sigma_0}_{s,t}\big|_{\tilde{Z}^{\sigma_0}}\]
            and $(\tilde{Z}^{\sigma}, (\tilde{g}^{\sigma}_{s,t})_{s\in\sigma, t\in [0,1]}, (\tilde{\psi}_s^{\sigma})) := (Z^{\sigma}, (g^{\sigma}_{s,t})_{s\in\sigma, t\in [0,1]}, (\psi_s^{\sigma}))$ for all $\sigma\neq\sigma_0$. Therefore, it is sufficient to verify that $\tilde{\psi}_s^{\sigma}(\tilde{Z}^{\sigma_0})$ is a $G$-invariant subset of $\mathcal{M}_T^s$. But by Claim \ref{claim G-invariant mu},  $\mu(\text{Int}D^3(1))$ is $G$-invariant in $Z^{\sigma_0}$ with respect to the induced action by $\tilde{\psi}_s^{\sigma_0}$ for all $s\in\sigma_0$. This implies the assertion. \\

        \paragraph{\textbf{Case 2}: $\sigma_0$ is a maximal simplex of $K$.}\
         
            By  \cite[Lemma~7.22]{bamler_ricci_2019}, it is sufficient to verify that the new partial homotopy 
                \[\{(\tilde{Z}^{\sigma}, (\tilde{g}^{\sigma}_{s,t})_{s\in\sigma, t\in [0,1]}, (\tilde{\psi}_s^{\sigma}))\}_{\sigma\subset K}\]
            constructed therein is $G$-invariant. We first note that
                \[\tilde{Z}^{\sigma} := \begin{cases}
                    Z^{\sigma}\quad &\text{if}\ \sigma\neq \sigma_0\\
                    Z^{\sigma_0}\setminus\mu(\text{Int}D^3(1))\quad &\text{if}\ \sigma = \sigma_0
                \end{cases}
                ,\quad \tilde{\psi}_s^{\sigma_0}:= \psi_s^{\sigma_0}\big|_{\tilde{Z}^{\sigma_0}}.\]
            Hence as in the proof of the previous proposition, it follows from Claim \ref{claim G-invariant mu} that $\tilde{\psi}_s^{\sigma_0}(\tilde{Z}^{\sigma} )$ is $G$-invariant in $\mathcal{M}_T^s$. Therefore we only need to verify that $(\tilde{\psi}_s^{\sigma_0})_*\tilde{g}^{\sigma}_{s,t}$ are $G$-invariant metrics on  $\tilde{\psi}_s^{\sigma_0}(\tilde{Z}^{\sigma} )$.\\

        \paragraph{\textbf{Step 1}} We first verify that the family of metrics $(h_{s,t})_{s\in U, t\in [0,1]}$ on $D^3(1.99)$ constructed in  \cite[Claim~7.25]{bamler_ricci_2019} is $G$-invariant with respect to the group action induced by the embedding $\mu: D^3(2)\to Z^{\sigma_0}$. We follow the original proof by induction on dimensions of simplices of $K$. If $\tau_j\subset K$ is a simplex, since $((\psi_s^{\tau_j})_*g_{s,t}^{\tau_j}, \psi_s^{\tau_j}(Z^{\tau_j}))$ is $G$-invariant, it follows from Remark \ref{G-invariance of induced action} that $\nu_s^*(\psi_s^{\tau_j})_*g_{s,t}^{\tau_j}$ is $G$-invariant on $D^3(1.99)\cap \nu_s^{-1}(\psi_s^{\tau_j}(Z^{\tau_j})))$. We can then extend $\nu_s^*(\psi_s^{\tau_j})_*g_{s,t}^{\tau_j}$ to a metric $h_{s,t}^j$ on $D^3(1.99)$ that is compatible with the standard spherical structure and the equivariant properties, using Proposition \ref{prop:equiv-symm-met} in place of \cite[Prop~6.15]{bamler_ricci_2019} in the original proof.\\
     
        \paragraph{\textbf{Step 2}} Next, using  Proposition \ref{prop:equiv-round-met} in place of \cite[Prop~6.26]{bamler_ricci_2019} in the original proof, we see that the family of metrics $(h'_{s,t})_{s\in U, t\in [0,1]}$ on $D^3(1.99)$ constructed in  \cite[Claim~7.26]{bamler_ricci_2019} is $G$-invariant with respect to the group action induced by the embedding $\mu: D^3(2)\to Z^{\sigma_0}$. Moreover, the metric $k_{s,t}^{\sigma}$ in item (g) of  \cite[Prop~6.26]{bamler_ricci_2019} is $G$-invariant PSC-conformal on $\psi_s^{\sigma}(Z^{\sigma})$ thanks to Claim \ref{claim G-invariant mu} and Lemma \ref{lem equiv PSC conformal}.\\

  \paragraph{\textbf{Step 3}} Following the proof of \cite[Prop~6.29]{bamler_ricci_2019}, we fix a continuous family of diffeomorphisms  $(\Phi_u: D^3(1.99)\to D^3(1.99))_{u\in (0,1]}$ with the additional requirement that $\Phi_u(x) = f(x,u)x$ for some smooth function $f$ satisfying $f(h.x, u) = f(x, u)$ for all $h\in G$, $x\in D^3(1.99)$ and $u\in (0,1]$. Then we set $h''_{s,t} = \Phi^*_{1-(r_2)\delta_1(s)\delta_2(t)}h'_{s,t}$ exactly as in the original proof. We now modify the group action on $D^3(2)$ from the embedding $\mu: D^3(2)\to Z^{\sigma_0}$  to the embedding $\mu\circ\Phi_{1-(r_2)\delta_1(s)\delta_2(t)}: D^3(2)\to Z^{\sigma_0}$. For $\chi\in G$, the modified group action on $D^3(2)$ satisfies 
  \[\chi. y = (\Phi_{1-(r_2)\delta_1(s)\delta_2(t)})^{-1}(\chi. (\Phi_{1-(r_2)\delta_1(s)\delta_2(t)}(y))) \]
  for $y\in D^3(2)$, where $\chi. (\Phi_{1-(r_2)\delta_1(s)\delta_2(t)}(y))$ means the group action on $D^3(2)$ induced by the embedding $\mu: D^3(2)\to Z^{\sigma_0}$. Using this and the $G$-invariance of $h'_{s,t}$ from Step 2,
  we check that
  \begin{align*}
  \chi^*h''_{s,t} = (\Phi_{1-(r_2)\delta_1(s)\delta_2(t)}\circ\chi)^*h'_{s,t} &= (\chi\circ\Phi_{1-(r_2)\delta_1(s)\delta_2(t)})^*h'_{s,t}  = h''_{s,t}.
  \end{align*}
  Thus the metric is $G$-invariant with respect to the group action induced by the embedding $\mu\circ\Phi_{1-(r_2)\delta_1(s)\delta_2(t)}: D^3(2)\to Z^{\sigma_0}$. Without loss of generality, for each $s\in U\cap\sigma$, $t\in [0,1]$, we may then replace $\nu_s(D^3(1.99))$ with $\nu_s\circ\Phi_{1-(r_2)\delta_1(s)\delta_2(t)}(D^3(1.99))$ in item (g) of \cite[Prop~6.29]{bamler_ricci_2019}, subsequently the metric $(\nu_s)_*h''_{s,t}$ is $G$-invariant on $\nu_s(D^3(1.99))$ by  Remark \ref{G-invariance of induced action}. With Lemma \ref{lem equiv PSC conformal}, the metrics $\tilde{k}_{s,t}^{\sigma}$ on $\psi_s^{\sigma}(Z^{\sigma})$ constructed in item (g) of  \cite[Prop~6.29]{bamler_ricci_2019} is therefore $G$-invariant PSC-conformal.\\

        \paragraph{\textbf{Step 4}} Lastly, we define
        \begin{align*}
            \tilde{g}_{s,t}^{\sigma} = \begin{cases}
         g_{s,t}^{\sigma}\quad &\text{on}\ Z^{\sigma}\setminus(\psi_s^{\sigma})^{-1}(\nu_s(D^3(1.99))\ \text{if}\ s\in U\ \text{or on}\ Z^{\sigma}\ \text{if}\ s\notin U\\
         (\psi_s^{\sigma})^*(\nu_s)_*h''_{s,t}\quad &\text{on}\     (\psi_s^{\sigma})^{-1}(\nu_s(D^3(1.99))\ \text{if}\ s\in U
         \end{cases}
        \end{align*}
        for $s\in\sigma\subset K$, $t\in [0,1]$ as in the proof of  \cite[Lemma~7.22]{bamler_ricci_2019}. By the results in the previous steps, $(\tilde{\psi}_s^{\sigma})_*\tilde{g}_{s,t}^{\sigma}$ are thus $G$-invariant  metrics on $\tilde{\psi}_s^{\sigma}(\tilde{Z}^{\sigma})$. This finishes the proof.
    \end{proof}

\section{Equivariant contractibility}

\subsection{Proof of Theorem \ref{thm:2}}\

We first show that Theorem \ref{thm:2} can be reduced to the equivariant version of \cite[Lemma~8.5]{bamler_ricci_2019}. Such a reduction follows from the equivariant version of \cite[Lemma~8.4]{bamler_ricci_2019}, and the equivariant changes are summarized below:

\begin{claim}{\label{lem BK 19 8.4}} In \cite[Lemma 8.4]{bamler_ricci_2019},
    \begin{itemize}
    \item[(i)] The embedded disks $D'\subset D\subset U_{S_2}^s\cap\mathcal{M}_{t_2}^s$  in item (f) are $G$-invariant;
    \item[(ii)] In item (h), if $(Y, g_t^s)$ is a $G$-invariant 3-submanifold, then $(Y, g'^{,s}_t)$ is $G$-invariant PSC-conformal.
    \end{itemize}
\end{claim}
\begin{proof}[Proof of Claim \ref{lem BK 19 8.4}]
    For assertion (ii), since the $\mathcal{R}$-structure is $G$-invariant, it follows from Lemma \ref{lem equiv PSC conformal} that $(Y, g'^{,s}_t)$ is $G$-invariant PSC-conformal. Next, to verify assertion (i), we follow closely with the original proof.  Given any $\delta' > 0$, we can find $C_0$ sufficiently large and $\delta$ sufficiently small such that $(\mathcal{M}_{t_2}^s, g_{t_2}^s, x)$ is $\delta'$-close to the pointed Bryant soliton $(M_{Bry}, g_{Bry}, x_{Bry})$ at scale $\rho(x)$. If $\delta'$ is small enough, then the scalar curvature of $g'^{,s}$ attains a unique maximum at some point $x'$ in a metric ball of radius $0.01\rho(x)$. Let $D'$ be the union of spherical fibers intersecting the minimizing geodesic from $x$ to $x'$. Moreover $x, x'\in U_{S_2}$ and $\{x'\}$ is a singular fiber. Since the spherical structure $\mathcal{S}^s$ is $G$-invariant, it follows that $D'$ is $G$-invariant.  
\end{proof}

Hence by using Claim  \ref{lem BK 19 8.4} together with \cite[Lemma 8.4]{bamler_ricci_2019}, it is sufficient to verify that the partial homotopy $\{(Z^{\sigma}, (g^{\sigma}_{s,t})_{s\in\sigma, t\in [0,1]}, (\psi_s^{\sigma})\}_{\sigma\subset K}$ constructed in \cite[Lemma 8.5]{bamler_ricci_2019} is $G$-invariant. For this purpose,  we need the following equivariant version of \cite[Lemma 8.6]{bamler_ricci_2019}: 

\begin{claim}{\label{lem BK 19 8.6}} After passing to a simplicial refinement of $\mathcal{K}'$ and modifying the $G$-invariant partial homotopy 
 $$\{(Z^{\sigma}, (g^{\sigma}_{s,t})_{s\in\sigma, t\in [0,1]}, (\psi_s^{\sigma})\}_{\sigma\subset K},$$ the additional data assumed in \cite[Lemma 8.6]{bamler_ricci_2019} further satisfy:
        \begin{itemize}
        \item The continuous family of embeddings $(\hat{\psi}_s^{\sigma}: \hat{Z}^{\sigma}\to \mathcal{M}_T^s)_{s\in\sigma}$ is  $G$-invariant;
        \item The continuous family of embeddings $(\nu_{s,j}^{\sigma}: D^3\to \mathcal{M}_T^s)$ is  $G$-invariant
        \end{itemize}
  for any $\sigma\in \mathcal{K}'$.

\end{claim}

\begin{proof}[Proof of Claim \ref{lem BK 19 8.6}]
    Following the construction in the original proof, it suffices to verify that the constructions in \cite[Lemma 8.7]{bamler_ricci_2019} and \cite[Lemma 8.8]{bamler_ricci_2019} are equivariant. Firstly, in the proof of \cite[Lemma 8.7]{bamler_ricci_2019} the set $Y\subset M_T^{s_0}\setminus\psi_{s_0}^{\sigma}(\text{Int}\ Z^{\sigma})$ consists of points satisfying  the curvature scale $\rho > 2r_k$. By $G$-invariance of the metrics $g_{s_0, t}^{\sigma}$, it follows that $Y$ is $G$-invariant. By the same reason, the embeddings $\hat{\psi}_s^{s_0, \sigma}(\hat{Z}^{\sigma})$ constructed in the original proof are also $G$-invariant. 

Next, in the proof of \cite[Lemma 8.8]{bamler_ricci_2019} the sets $Y_j\subset M_T^{s_0}$ are union of all open disks consisting of spherical fibers and satisfying the curvature scale $\rho\leq 3\Lambda r_k$. Hence by the $G$-invariance of spherical structure and the metrics, the closures $\bar{Y}_j$ are also $G$-invariant. Since the resulting embedd disks are chosen such that $\nu_{s_0, j}^{s_0, \sigma}(D^3) = \bar{Y}_j$, these embeddings are also $G$-invariant.
\end{proof}

We now modify the given $G$-invariant partial homotopy 
$$\{(Z^{\sigma}, (g^{\sigma}_{s,t})_{s\in\sigma, t\in [0,1]}, (\psi_s^{\sigma})\}_{\sigma\subset K}$$
at time $T$ relative to $L$ as in \cite[Subsection 8.4]{bamler_ricci_2019}. Using Proposition \ref{equivariant modification of enlarging a partial homotopy} and \ref{equivariant modification of removing a disk from a partial homotopy}, which are the equivariant versions of \cite[Proposition 7.9 and 7.21]{bamler_ricci_2019}, and the equivariant data $\hat{Z}^{\sigma}, \iota^{\sigma}, (\hat{\psi}_s^{\sigma})_{s\in\sigma}, (\nu_{s,j}^{\sigma})$ in Claim \ref{lem BK 19 8.6}, we may proceed the induction argument in \cite[page 103]{bamler_ricci_2019} to obtain a $G$-invariant partial homotopy
$$\{(\tilde{Z}^{\sigma}, (\tilde{g}^{\sigma}_{s,t})_{s\in\sigma, t\in [0,1]}, (\tilde{\psi}_s^{\sigma})\}_{\sigma\subset K}$$
at time $T$ relative to $L$ that satisfies the assertions in \cite[Lemma 8.15]{bamler_ricci_2019}. With this, we can complete the proof by constructing a partial homotopy at time $T - \Delta T$. We apply Proposition \ref{equivariant modification of moving a partial homotopy backwards in time} in place of \cite[Proposition 7.6]{bamler_ricci_2019} to the $G$-invariant partial homotopy $\{(\tilde{Z}^{\sigma}, (\tilde{g}^{\sigma}_{s,t})_{s\in\sigma, t\in [0,1]}, (\tilde{\psi}_s^{\sigma})\}_{\sigma\subset K}$. Therefore, the new partial homotopy $\{(\tilde{Z}^{\sigma}, (\bar{g}^{\sigma}_{s,t})_{s\in\sigma, t\in [0,1]}, (\bar{\psi}_s^{\sigma})\}_{\sigma\subset K}$ at time $T - \Delta T$ constructed in \cite[Subsection 8.4]{bamler_ricci_2019} is still $G$-invariant. \\

\subsection{Proof of Corollary \ref{equiv smale conj}}\

Inspired by the proof of \cite[Lemma 2.2]{bamler_ricci_2017}, we first prove the following fibration lemma for any closed subgroup $G$ of $SO(n + 1)$.

\begin{lemma}
    Let $\Met_{K\equiv 1}(S^n, G)\ (n \geq 2)$ denote the space of all $G$-invariant metrics on $S^n$ with constant sectional curvature $1$, and then $\Met_{K\equiv 1}(S^3, G)$ is homeomorphic to the orbit space $\Diff(S^n, G)/\Isom(S^n, G)$. Therefore, we have a natural fiber bundle,
    \[
        \Isom(S^n, G) \rightarrow \Diff(S^n, G) \xrightarrow{\pi} \Met_{K \equiv 1}(S^n, G)\,.
    \]
\end{lemma}
\begin{proof}

    It suffices to show that the action $\Diff(S^n, G) \acts \Met_{K\equiv 1}(S^n, G)$ by pushforward is transitive, with stabilizer $\Isom(S^n, G)$. Then the space $\Met_{K\equiv 1}(S^n, G)$ is homeomorphic to the orbit space $\Diff(S^n, G)/\Isom(S^n, G)$, and the fibration follows.

    Indeed, we only need to prove that for any metric $g \in \Met_{K\equiv 1}(S^n, G)$, there exists $\varphi \in \Diff(S^n, G)$, such that
    \[
        \varphi_\sharp(g_{\mathrm{round}}) = g\,,
    \]
    where $g_{\mathrm{round}}$ is a fixed round metric.

    By the existence of monodromy map, we know that there exists $\psi \in \Diff(S^n)$ such that
    \[
        \psi_\sharp(g_{\mathrm{round}}) = g\,. 
    \]
    However, a priori, $\psi$ is not necessarily $G$-invariant, and there is a large freedom to choose $\psi$. More precisely, for any $r \in \Isom(S^n, g_{\mathrm{round}})=SO(n+1)$, $\psi \comp r$ also satisfies the identity above. The transitivity follows from the next claim.

    \begin{claim}
        There exists a rotation $r \in SO(n+1)$, such that $\psi \comp r$ is $G$-invariant.
    \end{claim}
    \begin{proof}
        Note that $\psi \comp r$ is $G$-invariant if and only if for any $h \in G$, we have
        \[
            h \comp \psi \comp r = \psi \comp r \comp h\,,
        \]
        i.e.,
        \[
            r \comp h \comp r^{-1} = \psi^{-1} \comp h \comp \psi\,.
        \]

        We verify that for any $h \in G$, $\psi^{-1} \comp h \comp\psi$ is also a rotation of $S^n$, i.e., an isometry of $S^n$. Indeed, for any pair of points $p, q \in S^n$, we have
        \[
            \begin{aligned}
                d_{g_0}\left((\psi^{-1} \comp h \comp \psi)(p), (\psi \comp h \comp \psi^{-1})(q)\right) &= d_{g}\left((h\comp \psi)(p), (h \comp \psi)(q)\right)\\
                    &= d_{g}(\psi(p), \psi(q))\\
                    &= d_{g_0}(p, q)\,,
            \end{aligned}
        \]
        where the second identity follows from the $G$-invariance of $g$. Therefore, $\psi^{-1} \comp h \comp\psi$ is a rotation.

        Now, we can let $\rho_1, \rho_2: G \rightarrow GL_{n+1}(\mathbb{R})$ be two representations defined by
        \begin{align}
            \rho_1(h) v_1 &:= h(v_1)\,,\\
            \rho_2(h) v_2 &:= (\psi^{-1} \comp h \comp \psi)(v_2)\,,
        \end{align}
        for every $h \in G$, $v_1 \in S_1 = S^n$ and $v_2 \in S_2 = S^n$. Moreover, by definitions of $\rho_1$ and $\rho_2$, $S_1$ is $G$-diffeomorphic to $S_2$ via $\psi$.

        It follows from Theorem 4.3 in \cite{MR520507} that $S_1$ is $G$-isometric to $S_2$, so there exists a rotation $r \in SO(n+1)$, such that
        \[
            r \comp h \comp r^{-1} = \psi^{-1} \comp h \comp \psi\,.
        \]
        For completeness, we will sketch the proof here. 

        Let's first recall a well-known spherical rigidity theorem by G. de Rham \cite{de_rham_reidemeisters_1964}:
        
        \textit{If two rotations of $S^n$ are diffeomorphic, then they are isometric.}

        This implies that for any $h \in G$, there exists a rotation $r_h \in SO(n+1)$, such that 
        \[
            r_h \comp h \comp r^{-1}_h = \psi^{-1} \comp h \comp \psi\,.
        \]
        Even though the choice of $r_h$ possibly depends on $h$, it follows that $\rho_1$ and $\rho_2$ have the same character:
        \[
            \chi_{\rho_1}(h) = \chi_{\rho_2}(h)\quad \forall h \in G\,.
        \]

        By the representation theory of compact Lie group, the representations $S_1$ and $S_2$ are isomorphic.
    \end{proof}
\end{proof}

\begin{corollary}
    The inclusion map $\Isom(S^3, G) \hookrightarrow \Diff(S^3, G)$ is a homotopy equivalence if and only if $\Met_{\CC}(S^3, G)$ is weakly contractible.
\end{corollary}
\begin{proof}
    It follows from the fibration in the previous lemma that $\Isom(S^3, G) \hookrightarrow \Diff(S^3, G)$ is weakly homotopy equivalent $\iff$ $\Met_{\CC}(S^3, G)$ is weakly contractible.

    On the other hand, \cite[Proposition 1.5]{curtis_automorphism_1975} implies that $\Diff(S^3, G)$ is a Fr\'echet submanifold of $\Diff(S^3)$, and thus has the homotopy type of CW complexes. Note also that $\Isom(S^3, G)$ is a closed subgroup of $SO(4)$, and thus a compact Lie group. Hence, we have $\Isom(S^3, G) \hookrightarrow \Diff(S^3, G)$ is is weakly homotopy equivalent $\iff$ the inclusion is a homotopy equivalence.
\end{proof}

It is easy to see that $\Met_{\CC}(S^3, G) = \Met_{K\equiv 1}(S^3, G) \times (0, \infty)$, and thus, $\Met_{K \equiv 1}(S^3, G)$ is homotopy equivalent to $\Met_{\CC}(S^3, G)$. Therefore, it suffices to prove that $\text{Met}_{CC}(M, G)$ is either empty or weakly contractible. For this purpose, we follow the proof of \cite[Theorem 1.2]{bamler_ricci_2019} and verify that it can be carried over to the equivariant setting. We first need the equivariant version of \cite[Lemma 9.2]{bamler_ricci_2019}:

\begin{lemma}\label{equiv BK lem 9.2}
    Let $(M, g)$ be a $G$-invariant manifold.
    \begin{itemize}
    \item[(a)]If $M = S^3/\Gamma$ and $g$ is a conformally flat metric, then there is a unique $G$-invariant $\phi\in C^{\infty}(M,G)$ depending smoothly on $g$ such that $(M, e^{2\phi}g)$ is isometric to the standard round sphere and $\int e^{-\phi} d\mu_g$ is minimal.  
    \item[(b)] If $M$ is diffeomorphic to a quotient of the round cylinder and $g$ is conformally flat, then there is a unique $G$-invariant $\phi\in C^{\infty}(M)$  depending smoothly on $g$ such that $e^{2\phi}g$ is isometric to a quotient of the standard round cylinder.
    \end{itemize}
\end{lemma}

\begin{proof}[Proof of Lemma \ref{equiv BK lem 9.2}]
    It suffices to verify that the constructions in the proof of \cite[Lemma 9.2]{bamler_ricci_2019} are $G$-invariant. For assertion (a), suppose that the function $\phi$ constructed in \cite[Lemma 9.2(a)]{bamler_ricci_2019} is not $G$-invariant. Let us fix some $h\in G$. By the fact that $g$ is a $G$-invariant metric and that $(M, e^{2\phi}g)$ is isometric to the standard round sphere, with the metric 
    $g_h = e^{2(\phi\circ h)}h^*g =  e^{2(\phi\circ h)}g$  the Riemannian manifold $(M, g_h)$ is also isometric to the standard round sphere. This contradicts to the uniqueness of $\phi$. Assertion (b) holds for the same reason.
\end{proof}

Now, if $\text{Met}_{CC}(M, G)\neq\emptyset$ we can follow the proof of \cite[Theorem 1.2]{bamler_ricci_2019} to construct a homotopy of pairs $(\hat{\alpha}_t: (D^{k+1}, S^k)\to (\text{Met}(M, G), \text{Met}_{CC}(M, G)))_{t\in [0,1]}$ with Theorem \ref{thm:2} in place of \cite[Theorem 1.6]{bamler_ricci_2019}, where $\hat{\alpha}_1$ takes values in the space of conformally flat metrics on $M$. Finally using Lemma \ref{equiv BK lem 9.2} in place of \cite[Lemma 9.2]{bamler_ricci_2019} in the  proof of \cite[Theorem 1.2]{bamler_ricci_2019}, we obtain a family of metrics that defines a null-homotopy in $\text{Met}_{CC}(M, G)$. It follows that $\text{Met}_{CC}(M, G)$ is weakly contractible. 

\section{Topology of equivariant manifolds with positive scalar curvature}

For simplicity, we assume that $G$ is a connected compact Lie group, and thus, the quotient $G/H$ is connected for any closed subgroup $H$. For a closed three manifold $M$, since $G$ is a subgroup of $\Diff(M)$, $G$ acts on $M$ effectively and diffeomorphically. By principal orbit type theorem, there exists a unique principal orbit type $G/H$ of $M$ (we follow the notion in \cite{bredonIntroductionCompactTransformation1972}), which is a homogenous space such that $G$ acts on $G/H$ transitively.

Note that $M$ admits a metric with positive scalar curvature, $M$ is necessarily a connected sum of spherical space forms and copies of $S^2 \times S^1$.\\

\paragraph{\textbf{Case 1}} $\dim(G/H) = 0$.

By connectedness, $G/H$ is a point and so is $G$. Hence, $M$ can be any connected sum of spherical space forms and copies of $S^2 \times S^1$.\\

\paragraph{\textbf{Case 2}} $\dim(G/H) = 1$.

    This is equivalent to the case where $G/H$ is $S^1$ and $G$ is $SO(2)$ acting effectively on $M$.

    In \cite{orlikActions3Manifolds1968}, P. Orlik and F. Raymond classified all $M$ with $SO(2)$ actions.
    \begin{enumerate}
        \item When $SO(2)$ has no fixed points, $M$ is either $S^3$, $S^2 \times S^1$, $N$ (non-orientable handle), $RP^2 \times S^1$, $L(p, q)$ (lens space), $K \times S^1$ (a trivial Klein bottle bundle over $S^1$), $KS$ (a non-trivial Klein bottle bundle over $S^1$), a quotient of $SO(3)$ or $Sp(1)$ by a finite, non-abelian, discrete subgroup, or a $K(\pi, 1)$ whose fundamental group has infinite cyclic center (provided it is not $T^3$).
        \item When $SO(2)$ has fixed points, $M$ is a connected sum of copies of $S^3$, $S^2\times S^1$, $N$, $RP^2 \times S^1$, $L(p, q)$.
    \end{enumerate}

    Since $M$ is orientable and admits positive scalar curvature, we can rule out some possibilities above.

    In the case without fixed points, since $M$ is orientable, $M$ could not be $N$, $RP^2 \times S^1$, $K\times S^1$ or $KS$. In addition, $M$ admitting positive scalar curvature could not be an aspherical space. In conclusion, $M$ is either a spherical space form or $S^2 \times S^1$.

    In the case with fixed points, the orientability implies that $M$ is a connected sum of $S^3$, copies of $L(p,q)$ and $S^2 \times S^1$.\\

\paragraph{\textbf{Case 3}} $\dim(G/H) = 2$.

    The classification of all topological actions of compact Lie groups on connected $(n+1)$-dimensional manifolds with $n$-dimensional orbits is originated from P. S. Mostert's work \cite{mostertCompactLieGroup1957} in 1957. A thorough list for the compact case with $n = 2$ can be found in W. D. Neumann's work \cite{neumann3DimensionalGManifolds2Dimensional1968} in 1968.

    Here we summarize two essential facts used in this case.
    \begin{enumerate}
        \item Since $G/H$ has dimension $2$, the orbit space $M^*=M/G$ has dimension $1$. It follows from \cite[Chapter IV, Theorem 8.1]{bredonIntroductionCompactTransformation1972} that $M^*$ is a $1$-manifold, i.e., either $S^1$ or $I = [0, 1]$;
        \item Since $G$ acts on $G/H$ effectively and transitively, either $G = T^2$ and $H = \set{\mathrm{Id}}$, or $G = SO(3)$ and $H = O(2)$ or $SO(2)$.
    \end{enumerate}
    
    These would lead to the general classification \cite[Table in p.221]{neumann3DimensionalGManifolds2Dimensional1968} that $M$ is $T^3$, $K \times S^1$, $KS$, $RP^2 \times S^1$, $N$, $S^2 \times S^1$, $L(p, q)$, or $RP^3 \sharp RP^3$.

    Because we assume that $M$ is orientable and admits positive scalar curvature, $M$ is either $S^2 \times S^1$, a lens space, or $P^3 \sharp P^3$. It also follows from the table \textit{loc. cit.} that $G$ is either $T^2$ or $SO(3)$ ($G \neq T^2$ when $M = RP^3 \sharp RP^3$).\\

\paragraph{\textbf{Case 4}} $\dim(G/H) = 3$.

    In this case, $M = G/H$ is a homogeneous space, and $G$ acts on $M$ effectively and transitively. To verify that $M$ is either a spherical space form or $S^2 \times S^1$, we choose a $G$-invariant metric with positive scalar curvature on $M$ and run Ricci flow from it. 

    The transitivity of $G$-action on $M$ implies that whenever there exists a singular point in $M$ at time $T$, i.e., $\lim_{t \rightarrow T}R_t(x) = \infty$, every point in $M$ has scalar curvature tending to $\infty$. In other words, $T$ is exactly the extinct time. Moreover, the blowup model at each point in $M$ at $T$ is a $\kappa$-solution. 

    The classification theorem of $\kappa$-solutions \cite[Theorem 3.2]{bamler_ricci_2019} implies that they should be either a round shrinking cylinder, its $\mathbb{Z}_2$-quotient, an isometric quotient of the round shrinking sphere, the Bryant soliton or rotationally symmetric $S^3$, or $RP^3$. The transitivity of $G$-action again implies that they could not be the Bryant soliton as it has a unique center. Hence, one can summarize from the orientability that $M$ is either a spherical space form or $S^2 \times S^1$.
    
    In the spherical space form case, it follows from \cite[Chap 2: Corollary 2.7.2]{wolf_spaces_2011} that $M$ is either $S^3$, $RP^3$, $L(m, 1)\ (m > 2)$, $S^3/D^*_m\ (m > 2)$, $S^3/T^*$ or $S^3/O^*$.

    These complete the proof of Theorem \ref{thm:3}.

\appendix
\section{Equivariant version of results in \cite[Sect. 6]{bamler_ricci_2019}}
    Here, we shall collect some necessary modifications on the results from \cite[Sect. 6]{bamler_ricci_2019}.
    
    \subsection{Equivariant PSC-conformal metrics}\label{app:1}
        In this subsection, let $(M, g)$ be a $3$-manifold with boundary, and suppose that there is a compact Lie group $G$ such that $(M, g)$ is $G$-invariant. We say that $(M, g)$ is \textbf{$G$-invariant PSC-conformal} if it is PSC-conformal and the corresponding conformal factor $w$ is $G$-invariant. 
        \begin{lemma}\label{lem equiv PSC conformal}
             A $G$-invariant manifold $(M, g)$ is PSC-conformal if and only if it is $G$-invariant PSC-conformal.
        \end{lemma}
        \begin{proof}
            It is well-known that one can construct a unit-volume biinvariant Riemannian metric on the compact Lie group $G$. Suppose that $(M, g)$ is PSC-conformal with the corresponding function $w$. let's take the average action of $G$ over $w$ to obtain 
            \begin{equation}
                \tilde{w} = \int_{h \in G} h^*(w) \mathrm{d}h\,,
            \end{equation}
            and it suffices to show that the $G$-invariant function $\tilde w$ satisfies (1)-(3) of \cite[Lemma~6.4]{bamler_ricci_2019}.
            
            For (1), this obviously follows from the $G$-invariance of $(M, g)$, since $\Delta$ commutes with each $h^*$ and $\tilde w$ is a linear combination of $h^*(w)$.
            
            For (2), this also follows from the $G$-invariance of $(M, g)$.
            
            For (3), by (2) and $G$-invariance, $h^*(w) = w$ on $\partial M$ for any $h \in G$. Therefore, we have
            \begin{equation}
                \begin{aligned}
                    (\nu_{\partial M} \tilde w^4)g|_{\partial M} &= \nu_{\partial M}  \left(\int_{h \in G} h^*(w)\right)^4g|_{\partial M}\\
                    &= 4 \left(\int_{h\in G} h^*(w)\right)^3 \nu_{\partial M} \left(\int_{h\in G} h^*(w)\right)g|_{\partial M}\\
                    &= 4 w^3 \nu_{\partial M} \left(\int_{h\in G} h^*(w)\right)g|_{\partial M}\\
                    &= \int_{h\in G} \left(4(h^*(w))^3 \nu_{\partial M} (h^*(w))\right)g|_{\partial M}\\
                    &= \int_{h\in G}\left(h_*(\nu_{\partial M}) (h^*(w))^4 \right)g|_{\partial M}\\
                    &= \int_{h\in G} h^*\left(\nu_{\partial M}(w)^4 \right) g|_{\partial M}\\
                    &= \int_{h\in G} h^*\left(\nu_{\partial M}(w)^4 g|_{\partial M}\right)\\
                    &= \int_{h\in G} h^*\left(A_{\partial M}\right)\\
                    &= A_{\partial M}\,,
                \end{aligned}
            \end{equation}
            where we use the facts that $h_*(\nu_{\partial M}) = \nu_{\partial M}$ and $h^*(A_{\partial M}) = A_{\partial M}$.
        \end{proof}

    \subsection{Extending equivariant symmetric metrics}
        \begin{proposition}\label{prop:equiv-symm-met}
            In \cite[Proposition~6.15]{bamler_ricci_2019}, $Y$ needs not be connected, and instead, suppose additionally that $M$, $Y$, $(\mathcal{S}^s)_{s \in D^n}$, $(g^1_s)$, $(g^2_{s, t})$ and $(g^3_{s,t})$ are all $G$-invariant for some finite group $G$. Then $(h_{s, t})$ are $G$-invariant.
        \end{proposition}
        \begin{proof}
            It follows closely from the original proof.

            Under the additional equivariant conditions, one can verify that the open subset $V$ constructed in \cite[Lemma~6.2]{bamler_ricci_2019} can be taken $G$-invariant. As in the original proof, we can replace $M$ by $V$, identify $\mathcal{S}^s$ with $\mathcal{S}$ and remove the family $(g^3_{s,t})$ together with Assumptions (ii), (iii) and Assertion (c).

            \begin{description}
                \item[Case 1. $\partial Y = \emptyset$] In this case, $M \equiv Y$, so we can simply put $h_{s,t} = g^1_s$.
                \item[Case 2. $\partial Y \neq \emptyset$]

                    By \cite[Lemma~6.1]{bamler_ricci_2019}, $Y$ is diffeomorphic to either a disjoint union of multiple $S^2 \times [0, 1], (S^2\times [-1, 1])/\mathbb{Z}_2$ and $D^3$. Since $G$-actions only map a component of $\partial Y$ to another, we can remove collar neighborhoods of $\partial Y$ from $Y$ equivariantly, the remaining $G$-invariant region that is a union of fibers denoted by $Z$.

                    Following the arguments in the original proof, it suffices to prove the following equivariant lemma.

                    \begin{lemma}
                        In \cite[Lemma~6.16]{bamler_ricci_2019}, let $(M, Y)$ be a disjoint union of $(S^2 \times [-2, 2], S^2 \times [-1, 1])$ with $G$-invariance. If we assume additionally that $\mathcal{S}$, $(g^1_s)$ and $(g^2_{s, t})$ are $G$-invariant, then the generated metrics $(h^1_{s,t})$ are $G$-invariant as well.
                    \end{lemma}
                    \begin{proof}
                        Note that each connected component $Y'$ of $\overline{Y \backslash Z}$ has $\operatorname{Stab}_G(Z)$ fixing fibers of spherical structures, so the lemma above follows directly from \cite[Lemma~6.17]{bamler_ricci_2019} and an elementary $G$-invariant extension.
                    \end{proof}
            \end{description}

            In both cases, the generated metrics $(h_{s,t})$ are $G$-invariant.
        \end{proof}

    \subsection{Extending equivariant symmetric metrics on the round sphere}
        \begin{proposition}
            In \cite[Proposition~6.18]{bamler_ricci_2019}, if $(M, g^*)$, $(k_{s,t})$ and $(\mathcal{S}^s)_{s \in A}$ are $G$-invariant, then so are the generated metrics $(h_{s,t})$.
        \end{proposition}
        \begin{proof}
            Similarly, we only need to check the two cases $E = \emptyset$ or $A = \Delta^n$.
            \begin{description}
                \item[Case 1. $E = \emptyset$] In this case, $h_{s, t} = \lambda^2_0(s, t) g^*$ is definitely $G$-invariant.
                \item[Case 2. $A = \Delta^n$] The construction of $(h_{s,t})$ is a direct consequence of out previous proposition, Proposition \ref{prop:equiv-symm-met}.
            \end{description}
        \end{proof}
        
    \subsection{The equivariant rounding metrics}
        Suppose that $G$ is a compact Lie group acting on $B^3$ preserving the standard spherical structures, we have the following equivariant result.

        \begin{proposition}\label{prop:equiv-round-met}
            In \cite[Proposition~6.26]{bamler_ricci_2019}, if $(h_s)_{s \in X} \in \Met(B^3, G)$ and $(w_s)_{s \in X_{\PSC}}$ are $G$-invariant, then $(h'_{s, u})_{s \in X, u\in [0, 1]}$ are $G$-invariant.
        \end{proposition}
        \begin{proof}
            Since the construction of conformal exponential maps is via an ODE, one can easily check that $\varphi$ constructed in \cite[Page 69]{bamler_ricci_2019} is $G$-invariant and so is the conformal exponential map. One can follow the proof of \cite[Proposition 6.26]{bamler_ricci_2019} and check that each step preserves the equivariance, except for that in the proof of \cite[Lemma 6.28]{bamler_ricci_2019}, we replace $S^+_3$ by the normalizer $N_{G_0}(S^+_3)$ and $S_3$ by the normalizer $N_{G_0}(S_3)$, where $G_0$ is the stabilizer group at the origin.
        \end{proof}

\bibliography{reference}
\end{document}